\documentclass[12pt]{amsart}
\usepackage[margin=1.5in]{geometry}
\usepackage{mathptmx,amsmath,amsthm}
\theoremstyle{plain}
\newtheorem{theorem}{Theorem}[section]
\newtheorem{proposition}[theorem]{Proposition}
\newtheorem{corollary}[theorem]{Corollary}
\newtheorem{lemma}[theorem]{Lemma}
\newtheorem{conjecture}[theorem]{Conjecture}
\newtheorem{condition}[theorem]{Condition}

\theoremstyle{definition}
\newtheorem{definition}[theorem]{Definition}

\theoremstyle{remark}

\allowdisplaybreaks

\begin{document}

\title[p-modular Hadamard matrices]{Hadamard matrices modulo p and small modular Hadamard matrices}
\author{Vivian Kuperberg}
\address{Cornell University}
\email{vzk2@cornell.edu}

\begin{abstract}
We use modular symmetric designs to study the existence of Hadamard matrices modulo certain primes. We solve the $7$-modular and $11$-modular versions of the Hadamard conjecture for all but a finite number of cases. In doing so, we state a conjectural sufficient condition for the existence of a $p$-modular Hadamard matrix for all but finitely many cases. When $2$ is a primitive root of a prime $p$, we conditionally solve this conjecture and therefore the $p$-modular version of the Hadamard conjecture for all but finitely many cases when $p \equiv 3 \pmod{4}$, and prove a weaker result for $p \equiv  1 \pmod{4}$.  Finally, we look at constraints on the existence of $m$-modular Hadamard matrices when the size of the matrix is small compared to $m$.
\end{abstract}
\maketitle

\section{Introduction}

Hadamard matrices have many applications in mathematics and in signal and data processing (see \cite{hor07}). A real Hadamard matrix of size $n$ is an $n \times n$ $(\pm 1)$-matrix $H$ such that $HH^\top = nI$. In particular, the rows and columns of $H$ are orthogonal. We can generalize this idea to modular Hadamard matrices; an $m$-modular Hadamard matrix of size $n$ is an $n \times n$ matrix $H$ with entries $\pm 1$ satisfying that $HH^\top \equiv nI \pmod{m}$. We write that $H$ is an $MH(n,m)$. Modular Hadamard matrices were introduced by Marrero and Butson in \cite{mar73}; further results were achieved in \cite{heb76} and \cite{mar75}. Motivated by the Hadamard conjecture, which states that Hadamard matrices of size $4n$ exist for all $n$, it has been fully determined for which $n$ an $MH(n,m)$ exists when $m = 2,3,4,5,6,8,12$. The authors in \cite{eli05} further prove that the Hadamard conjecture holds for $32$-modular Hadamard matrices.

In this paper we begin by recalling established results for modular Hadamard matrices and modular symmetric designs. We then use a direct sum construction to prove the $7$-modular Hadamard conjecture for all but finitely many cases and provide a conditional construction for certain $p$-modular Hadamard matrices. Finally, we use combinatorial techniques to prove nonexistence for small modular Hadamard matrices.

Throughout this paper, we say that a Hadamard matrix $H$ is \emph{normalized} if all entries in its first row and first column are $+1$. Any Hadamard matrix can be normalized by multiplying rows and columns by $-1$, so unless otherwise specified we will assume that all Hadamard matrices mentioned are normalized.  Moreover, $J$ will refer to the matrix of all $1$'s, $(a,b)$ will represent the gcd of $a$ and $b$, and, unless otherwise specified, $m,n,$ and $k$ will be nonnegative integers.

\section{Modular Hadamard Matrices}
\label{mhm}
A few basic results, all presented in \cite{lee13}, allow us to completely decide the existence of Hadamard matrices modulo 2, 3, 4, 6, 8, and 12. 

\begin{lemma}\cite{lee13}
Assume $n \geq 3$, and let $H$ be an $MH(n,m)$. Then $(m,4) \vert n$. If $n$ is odd and $n \not\equiv 0 \pmod{m}$, then let $1 \leq r \leq m-1$ such that $r \equiv 2^{\phi(m) - 2}n \pmod{m}$; we know that $n \geq 4r$.
\label{gcdmn} 
\end{lemma}

\begin{lemma}\cite{lee13}
Let $H$ be an $MH(n,m)$, with $(n,m) = 1$ and $n$ odd. Then $n$ is a quadratic residue of $m$.
\label{oddqrs}
\end{lemma}
\begin{proof}
Since $HH^\top \equiv nI \pmod{m}$, $(\mathrm{det}\: H)^2 \equiv n^n \pmod{m}$, so $n$ can not be both odd and a nonresidue of $m$.
\end{proof}

\begin{lemma}\cite{lee13}
If $n \equiv 0 \pmod{m}$ or $n \equiv 4 \pmod{m}$, then there exists an $MH(n,m)$.
\label{04modm}
\end{lemma}
\begin{proof}
If $n \equiv 0 \pmod{m}$, then the matrix $J$ is an $MH(n,m)$. If $n \equiv 4 \pmod{m}$, then the matrix $J-2I$ is an $MH(n,m)$.
\end{proof}

\begin{lemma}\cite{lee13}
Let $H_1$ be an $MH(n_1,m_1)$ and $H_2$ be an $MH(n_2,m_2)$. Then $H_1 \otimes H_2$ is an $MH(n_{1}n_{2},(m_{1}m_{2},n_{1}m_{2},n_{2}m_{1}))$. 
\label{double}
\end{lemma}

The operation, described in Lemma \label{double} is most commonly done when one of the components is the real Hadamard matrix
\[F_2 = 
\begin{bmatrix}
1&1\\
1&-1
\end{bmatrix},\]
in which case it is called ``doubling."

\begin{theorem}
If $n \geq 3$, then \\
(a) an $MH(n,2)$ exists $\iff$ $n$ is even. \\
(b) an $MH(n,3)$ exists $\iff$ $n \not\equiv 5 \pmod{6}$. \\
(c) an $MH(n,4)$ exists $\iff$ $n \equiv 0 \pmod{4}$. \\
(d) an $MH(n,6)$ exists $\iff$ $n$ is even. \\
(e) an $MH(n,8)$ exists $\iff$ $n \equiv 0 \pmod{4}$. \\
(f) an $MH(n,12)$ exists $\iff$ $n \equiv 0 \pmod{4}$. \\
\end{theorem}
\begin{proof}
Cases (a) through (d) are addressed in \cite{lee13}. An $MH(n,8)$ can only exist by Lemma \ref{gcdmn} if $n \equiv 0,4 \pmod{8}$ and both of these are constructed in Lemma \ref{04modm}. By Lemma \ref{gcdmn}, an $MH(n,12)$ can only exist if $n \equiv 0,4,8 \pmod{12}$. The $n \equiv 0,4, \pmod{12}$ cases are constructed in Lemma \ref{04modm}; for the $n \equiv 8 \pmod{12}$ case, we can double an $MH(6k+4,6)$ for any $k \geq 0$ to get an $MH(12k+8,12)$.
\end{proof}

The study of modular symmetric designs was explored in \cite{lee13} to address the question of deciding the existence of $MH(n,m)$ matrices for $m = 5$. 

\section{Modular Symmetric Designs}

In much the same way that modular Hadamard matrices are a generalization of Hadamard matrices, we can generalize modular symmetric designs from symmetric designs, leading to the following definition.

\begin{definition}
Let $m,v \geq 2$ be integers. An \emph{$m$-modular symmetric design} is a $v \times v$ $(0,1)$-matrix $D$ satisfying that $DD^\top \equiv (k - \lambda) I + \lambda J \pmod{m}$ and that $DJ \equiv JD \equiv kJ$ for some integers $k,\lambda$, and where $J$ is the $v \times v$ matrix of all $1$'s. Such a design is denoted by its parameters $(v,k,\lambda;m)$.
\end{definition}

We may go between modular symmetric designs and modular Hadamard matrices by replacing the $-1$ entries with $0$ entries, and vice versa. The following definitions and lemmas provide conditions on this process.

\begin{definition}
The \emph{core} of a normalized modular Hadamard matrix $H$ is obtained by discarding its first row and first column and is denoted $C(H)$. We will denote by $D(H)$ the matrix $(C(H) + J)/2$, which is the core of the normalized matrix where $-1$ entries are replaced by $0$ entries.
\end{definition}

\begin{lemma}\cite{lee13}
\label{makedesign}
Let $H$ be an $MH(n,m)$ with $(n,m) = 1$ and $n,m \geq 3$. The design $D(H) = (C(H) + J)/2$ is a $(n-1,2^{\phi(m) - 1}(n-2),2^{\phi(m)-2}(n-4);m)$ design.
\end{lemma}

The reverse transformation also holds under certain conditions; in particular, two modular symmetric designs occasionally generate a larger modular Hadamard matrix.

\begin{definition}
Let $D_1$ be a $(v_1,k_1,\lambda_1;m)$ design and $D_2$ be a $(v_2,k_2,\lambda_2;m)$ design. Then \\
\[\begin{bmatrix}
D_1 & J \\
J^\top & D_2
\end{bmatrix}\] \\
is defined as the \emph{direct sum} of $D_1$ and $D_2$, and denoted by $D_1 \oplus D_2$. 
\end{definition}

While $D_1 \oplus D_2$ is not necessarily a modular symmetric design, the following is true.

\begin{lemma}
\label{constraints}
\cite{lee13}
$2(D_1 \oplus D_2) - J$ is an $MH(v_1 + v_2,m)$ if and only if
\[v_1 + v_2 \equiv 4(k_1 - \lambda_1) \equiv 4(k_2 - \lambda_2) \equiv 2(k_1 + k_2) \pmod{m}.\]
\label{dsum}
\end{lemma}

Assuming $m$ is odd, the constraints of Lemma \ref{dsum} imply that if $H$ is an $MH(n,m)$, then $2(D(H) \oplus D_2) - J$ is an $m$-modular Hadamard matrix if and only if the parameters $(v_2, k_2, \lambda_2; m)$ of the design $D_2$ satisfy the following equivalences:
\begin{align*}
v_2 &\equiv 1 \pmod{m} \\
k_2 &\equiv 1 \pmod{m} \\
\lambda_2 &\equiv 2^{\phi(m) - 2}(4-n) \pmod{m}.
\end{align*}

The authors in \cite{lee13} use the above lemmas to prove the following theorem.

\begin{theorem}\cite{lee13}
There exists an $MH(n,5)$ if and only if $n \not\equiv 3,7 \pmod{10}$ and $n \neq 6,11$.
\end{theorem}

\section{The Hadamard Conjecture modulo 7}

Constructing $7$-modular Hadamard matrices is more difficult than those discussed above. Several cases can still be handled using modular Hadamard matrices alone.

\begin{proposition}
\label{mod7s047811}
There exists an $MH(n,7)$ if $n \equiv 0, 4, 7, 8, 11 \pmod{14}$. There does not exist an $MH(n,7)$ if $n \equiv 3, 5, 13 \pmod{14}$.
\end{proposition}
\begin{proof}Just as in Lemma \ref{04modm}, $J$ is an $MH(7k + 7,7)$ and $J - 2I$ is an $MH(7k+4,7)$. An $MH(14k+8,7)$ can be obtained by doubling an $MH(7k+4,7)$.
Meanwhile, Lemma \ref{oddqrs} shows that if $n \equiv 3,5, 13 \pmod{14}$, then no $MH(n,7)$ exists.
\end{proof}

Using the direct sum of modular symmetric designs, we can fully say when an $MH(7k + 1,7)$ exists with the exception of an $MH(15,7)$ and an $MH(29,7)$.

\begin{proposition}
If $n \equiv 1 \pmod{7}$ and $n \neq 15, 29$, then an $MH(n,7)$ exists.
\end{proposition}
\begin{proof}
Using the construction of Proposition \ref{mod7s047811}, we know that for all $k$, an $MH(14k+8,7)$ exists. We can use that modular Hadamard matrix with Lemma \ref{makedesign} to generate a modular Hadamard design with parameters $(14k+7, 3, 1;7)$. Taking its direct sum with a design of parameters $(36,15,6)$, a Menon design of size $36$ (see \cite{ion06}), we obtain an $MH(14k+43,7)$, giving us an $MH(14k+1,7)$ if $k \geq 3$ and thus an $MH(n,7)$ for all $n \neq 15,29$ and $n \equiv 1 \pmod{7}$.
\end{proof}

\begin{proposition}
An $MH(7k+2,7)$ exists if $7k+2 \geq 52565$.
\end{proposition}
\begin{proof}
We double an $MH(14k+8,7)$ and an $MH(14k+43,7)$ to get an $MH(28k+16,7)$ and an $MH(28k+86,7)$, respectively. We then take the direct sum of the cores of each of these with a $(52480,5832,648)$ design (see family 12 in \cite{ion06}) to get an $MH(28k+52495,7)$ and an $MH(28k+52565,7)$. All sizes congruent to $2 \pmod{7}$ and $\geq 52565$ are represented in one of these four families.
\end{proof}

Constructions for $7$-modular Hadamard matrices of other sizes are more complicated. They depend on the idea that we can iteratively generate new modular Hadamard matrices by using an appropriate choice of design, as described in the following lemma.

\begin{lemma}
\label{iteratelemma}
Let $H$ be an $MH(n,m)$, and let $D$ be a design with parameters $(v,k,\lambda;m)$, where $v \equiv 1 \pmod{m}$, $k \equiv 1 \pmod{m}$, and $\lambda \equiv 2^{\phi(m) - 2}(4-n) \pmod{m}$. Then for all integers $l \geq 0$, an $MH(n + l(v-1), m)$ exists.
\end{lemma}
\begin{proof}
The $l = 0$ case is trivial. We then proceed inductively; assume that an $MH(n + (l-1)(v-1), m)$ exists, and call it $H_{l-1}$. Since $v \equiv 1 \pmod{m}$, $n + (l-1)(v-1) \equiv n \pmod{m}$. By Lemma \ref{constraints}, our design $D$ has parameters such that $2(D(H) \oplus D) - J$ is a modular Hadamard matrix, so because the sizes of $H$ and $H_{l-1}$ are congruent modulo $m$, $2(D(H_{l-1}) \oplus D) - J$ is also a modular Hadamard matrix. Specifically, it is an $MH(n + (l-1)(v-1) - 1 + v, m)$, or an $MH(n + l(v-1),m)$. Thus, an $MH(n + l(v-1),m)$ exists, so by induction such a matrix exists for all $l \geq 0$.
\end{proof}

This iterative process is used to varying extents to generate $7$-modular Hadamard matrices in the following results.

\begin{proposition}
An $MH(14k+6,7)$ exists if $14k+6 \geq 398$.
\end{proposition}
\begin{proof}
Real Hadamard matrices of size 12 are known to exist (see \cite{eli05}). Let $M_{12}$ be one of these. Then we can take the Kronecker product $MH(7k+4,7) \otimes M_{12}$ to get an $MH(84k+48,7)$. Using the design $(71,15,3)$ from \cite{ion06} and Lemma \ref{iteratelemma}, we create an $MH(84k+118,71)$, an $MH(84k + 188,7)$, an $MH(84k + 258,7)$, an $MH(84k + 328,7)$, and an $MH(84k + 398,7)$. All sizes congruent to $6 \pmod{14}$ and $\geq 398$ are represented in one of these six families.
\end{proof}

\begin{proposition}
An $MH(14k+10,7)$ exists if $14k+10 \geq 683294$.
\end{proposition}
\begin{proof}
For $k$ large enough that an $MH(7k+2,7)$ exists, we can take the product $MH(7k+2,7) \otimes M_{12}$ to get an $MH(84k+24,7)$. We can then take the design $(25439,12167,5819)$, from family 11 in \cite{ion06}, and use Lemma \ref{iteratelemma} to get an $MH(84k+25462)$ and an $MH(84k + 50900,7)$, in addition to the original $MH(84k + 24,7)$. We can further take the design $(1639,729,324)$, also from family 11 in \cite{ion06}, and use Lemma \ref{iteratelemma} once with each of these three designs to get, respectively, an $MH(84k + 27100,7)$, an $MH(84k + 52538,7)$, and an $MH(84k+1662,7)$. All sufficiently large $14k + 10$ will fall into one of these six families of sizes modulo $84$.
\end{proof}

\begin{proposition}
An $MH(14k+12,7)$ exists if $14k + 12 \geq 4.5 \times 10^{36}$.
\end{proposition}
\begin{proof}
\cite{eli05} contains a construction for a size-20 real Hadamard matrix, denoted $M_{20}$. If $k$ is large enough, we can take $MH(7k+2,7) \otimes M_{20}$ to get an $MH(140k+40,7)$. Using the design $(2185,729,243)$ of family 10 in \cite{ion06} and the technique from Lemma \ref{iteratelemma}, we can generate an $MH(140k+2224,7)$, an $MH(140k+4408,7)$, an $MH(140k+6592,7)$, and an $MH(140k+8776,7)$. For all large enough sizes $n$ with $n \equiv 5 \pmod{7}$ and $n \equiv 0 \pmod{4}$, or simply $n \equiv 12 \pmod{28}$, these constructions prove that an $MH(n,7)$ exists. 

We must therefore only account for the case when the size $n$ is equivalent to $26 \pmod{28}$. We can make this construction, for sufficiently large $n$, by simply taking the direct sum of the core of an $MH(n_0,7)$, with $n_0 \equiv 12 \pmod{28}$, as constructed in the previous paragraph, and an appropriate design of size congruent to $3 \pmod{4}$. We can use a design from family 10 of \cite{ion06}. These designs are of the form \[\left(1 + \frac{qr(r^m-1)}{(r-1)},r^m,\frac{r^{m-1}(r-1)}{q}\right),\] where $q$ and $r$ are prime powers. Choosing $q = 29$, $d = 5$, $r = 732541$, and $m = 6$, we have a design from this family that satisfies the necessary parameters for the direct sum construction. Thus all sufficiently large $n \equiv 12 \pmod{14}$ are sizes of some $7$-modular Hadamard matrix. The size of this design is a bit $4.5 \times 10^{36}$; matrices all sizes congruent to $12 \pmod{14}$ and $\geq 4.5 \times 10^{36}$\footnote{The exact bound is $4481157543653329008412788039740507382.$} are constructed by this proof.
\end{proof}

These results determine for which $n$ an $MH(n,7)$ exists with finitely many exceptions, and yield the following theorem.

\begin{theorem}
\label{donefor7}
The Hadamard conjecture modulo $7$ is true for $n \geq 4.5 \times 10^{36}$.
\end{theorem}

\section{p-modular Hadamard matrices}

The results in \cite{lee13} and Theorem \ref{donefor7} motivate the following conjecture.

\begin{conjecture}
Let $p$ be an odd prime. For all but finitely many $n$, an $MH(n,p)$ exists if and only if $n$ is even or a quadratic residue of $p$.
\label{allth}
\end{conjecture}

This has already been shown for $p = 3,5,7$, and Lemma \ref{oddqrs} proves the only-if direction of the conjecture for all $p$. Assuming certain number theoretical conjectures, we prove this conjecture for primes $p \equiv 3 \pmod{4}$ and such that $2$ is a primitive root of $p$. 

\begin{theorem}
Let $p \equiv 3 \pmod{4}$ be a prime such that $2$ is a primitive root of $p$, and let $r$ be any quadratic residue of $p$. For all but finitely many even values of $n$, with $n \equiv 2(1+r) \pmod{p}$, an $MH(n,p)$ exists.
\label{gentech}
\end{theorem}
\begin{proof}
Fix $p$ and $r$, and define $n_0 = 2(1+r)$. For every integer $j \geq 1$, an $MH(jp + 4,p)$ exists by Lemma \ref{04modm}; let $M_j$ be one such $MH(jp+4,p)$. Moreover, since $2$ is a primitive root of $p$, we know that there exists some smallest positive integer $i$ with $2^{i+2} \equiv n_0 \pmod{p}$. A real Hadamard matrix of size $2^i$ can be constructed via the doubling technique beginning with $F_2$; call the $2^i \times 2^i$ Hadamard matrix $F_{2^i}$. Then $M_j \otimes F_{2^i}$ is an $MH(2^ijp + 2^{i+2},p)$, so for all $n \equiv n_0 \pmod{p}$ where $n \equiv 2^{i+2} \pmod{2^ip}$ and $n \geq 2^ip + 2^{i+2}$, an $MH(n,p)$ exists. 

For every congruence class mod $2^ip$ with elements that are even and congruent to $n_0 \pmod{p}$, we would like to do the same; namely, we would like to show for each of these congruence classes that for all but finitely many elements $n$, an $MH(n,p)$ exists. If $i= 0$ we are done, given the construction above. If $i > 0$, each congruence class mod $2^ip$ is either even or odd. We can accomplish our goal in one fell swoop by finding a symmetric design $D$ of size $v$ congruent to $3 \pmod{4}$ and $1 \pmod{p}$, where $2(D(M_j) \oplus D) - J$ is a $p$-modular Hadamard matrix of size congruent to $n_0 \pmod{p}$. This constructs matrices of size $2^ipj + 2^{i+2} + v - 1$ for any $j \geq 0$, so by varying $j$ we have constructed $p$-modular Hadamard matrices for all but finitely many of the sizes in the congruence class $2^{i+2} + v - 1 \pmod{2^ip}$. Let $D$ be a design such that the sum $2(D(M_j) \oplus D) - J$ is a $p$-modular Hadamard matrix of size congruent to $n_0 \pmod{p}$, and call it $M'_j$. Then, because the constraints for the parameters of $D$ depend only on $n_0$ and $p$, $2(D(M'_j) \oplus D) - J$ is also a $p$-modular Hadamard matrix of size congruent to $n_0 \pmod{p}$. Defining $M^{(l)}_j$ as $2(D(M^{(l-1)}_j) \oplus D) - J$ for all $l \geq 0$, then for each fixed $l$, by varying $j$ we construct all but finitely many $p$-modular Hadamard matrices of size congruent to $2^{i+2} + l(v-1) \pmod{2^ip}$. However, $v-1$ is oddly even and divisible by $p$, so upon varying $l$ the values $2^{i+2} + l(v-1)$ achieve each even congruence class mod $2^ip$ that is congruent to $n_0 \pmod{p}$. Thus finding an appropriate choice of a design $D$ is sufficient to prove that for each congruence class of $2^ip$ with elements even and congruent to $n_0 \pmod{p}$, for all but finitely many elements $n$ of this congruence class, an $MH(n,p)$ exists.

Thus far our proof only depends on the primality of $p$ and the fact that $2$ is a primitive root; we will use this method in later arguments as well.

Consider designs $D$ in family 11 of \cite{ion06}, of the form
\[\left(1 + \frac{2q(q^m - 1)}{(q-1)},q^m, \frac{q^{m-1}(q-1)}{2}\right), \]
with $q$ an odd prime power and $m$ an integer. Then, with reference to the constraints on $v$ and those in Lemma \ref{dsum}, we know that the following holds:
\begin{align*}
1 + \frac{2q(q^m-1)}{(q-1)} &\equiv 1 \pmod{p}\\
1 + \frac{2q(q^m-1)}{(q-1)} &\equiv 3 \pmod{4}\\
q^m &\equiv 1 \pmod{p}\\
\frac{q^{m-1}(q-1)}{2} &\equiv 2^{\phi(p) - 2}(4-n_0)\pmod{p}.\\
\end{align*}
$\frac{2q(q^m-1)}{(q-1)}$ must be oddly even, so $\frac{q(q^m-1)}{(q-1)}$ must be odd. Thus, $q$ is odd, as is $\frac{q^m-1}{q-1} = q^{m-1} + \cdots + q + 1$, so $m$ is odd. Since $q^m \equiv 1 \pmod{p}$, the order of $q \pmod{p}$ must divide $m$; thus the order of $q$ is odd. $p \equiv 3 \pmod{4}$, so we can guarantee this condition exactly when $q$ is a quadratic residue of $p$. If $q \not\equiv 1 \pmod{p}$, we set $m$ equal to the order of $q$, and we automatically have satisfied the first three conditions. If $q \equiv 1 \pmod{p}$, we instead let $m$ be $p$ to satisfy the first three conditions.
The remaining constraint gives: 
\begin{align*}
\frac{q^{m-1}(q-1)}{2} &\equiv 2^{\phi(p) - 2}(4-n_0)\pmod{p}\\
\Leftrightarrow 2(1 - q^{m-1}) &\equiv 4-n_0 \pmod{p}\\
\Leftrightarrow n_0 &\equiv 2(q^{-1} + 1) \pmod{p}.\\
\end{align*}

$n_0 = 2(1+r)$, with $r$ a residue, so we can find some prime $q$ with $q \equiv r^{-1} \pmod{p}$ and define $m$ as above. Our choice of $D$ thus guarantees that all but finitely many $p$-modular Hadamard matrices of size congruent to $n_0 \pmod{p}$ and even exist.
\end{proof}

For later results, we define the following condition on a prime $p$.

\begin{condition}
For every $1 \leq \delta < p$, there exists an odd prime power $q \equiv 1 \pmod{p}$ and a $d \equiv \delta \pmod{p}$ such that $r = \frac{q^d - 1}{q-1}$ is a prime power with $r \equiv 1 \pmod{4}$.
\label{pps}
\end{condition}

Fixing a $1 \leq \delta < p$ and a prime $d \equiv \delta \pmod{p}$, and letting our choice of $q$ vary, this condition becomes a special case of Schinzel's hypothesis H (see \cite{sch58}). We can let $q(x) = px + 1$ and $r(x) = \frac{q(x)^d - 1}{q(x) -1}$ be polynomials; these polynomials are irreducible and there is no prime that divides their product at every value. Schinzel's hypothesis H states that there is some $x$ for which $q(x)$ and $r(x)$ are both prime; taking these as our $q$ and $r$, respectively, gives us the above condition. 
Assuming this condition, we have the following result.

\begin{theorem}
Let $p$ be a prime with $2$ as a primitive root and such that $p \equiv 3 \pmod{4}$ and satisfies Condition \ref{pps}. For all but finitely many $n$, an $MH(n,p)$ exists if $n$ is not an odd quadratic nonresidue of $p$.
\label{3mod4}
\end{theorem}
\begin{proof}
We will prove this in two stages; first, by proving that for all but finitely many even $n$, an $MH(n,p)$ exists, and second, by proving that for all but finitely many quadratic residues $n$, an $MH(n,p)$ exists.
Since $2$ is a primitive root of $p$, we can use the procedure from the proof of Theorem \ref{gentech} to construct all desired matrices of even size, provided that we have a design with appropriate parameters. We use designs from family 10 in \cite{ion06} of the form
\[\left(1 + \frac{qr(r^m-1)}{(r-1)},r^m,\frac{r^{m-1}(r-1)}{q}\right)\]
where $q$ and $r = \frac{q^d-1}{q-1}$ are both prime powers, and $m$ and $d$ are positive integers. We need only that the parameters of this design satisfy the constraints
\begin{align}
1+\frac{qr(r^m-1)}{(r-1)} \equiv 3 \pmod{4}\label{eq:1}\\
1+\frac{qr(r^m-1)}{(r-1)} \equiv 1 \pmod{p}\label{eq:2}\\
r^m \equiv 1 \pmod{p}\label{eq:3}\\
\frac{r^{m-1}(r-1)}{q} \equiv 2^{\phi(p)-2}(4-n) \pmod{p}.\label{eq:4}
\end{align}
Assume that we have a $q \equiv 1 \pmod{p}$. Then, constraint (\ref{eq:4}) becomes 
\[r^{m-1}(r-1) \equiv 2^{\phi(p) - 2}(4-n) \pmod{p}\]
which, given constraint (\ref{eq:2}), further reduces to
\[n \equiv 4r^{-1} \pmod{p}.\]
Note that since $r = q^{d-1} + \cdots + 1$, $r \equiv d \pmod{p}$. We choose $\delta \pmod{p}$ so that $n \equiv 4\delta^{-1} \pmod{p}$. Then, since $p$ satisfies Condition \ref{pps}, we can find a $d$ so that (\ref{eq:4}) is satisfied. If $r \equiv 1 \pmod{p}$, we can set $m = 2p$, meaning that $\frac{qr(r^m-1)}{(r-1)}$ will be congruent both to $2 \pmod 4$ and to $0 \pmod{p}$, so all four constraints are satisfied. Otherwise, we set $m = \phi(p) = p-1$. Since $r \equiv 1 \pmod{4}$ and $p \equiv 3 \pmod{4}$, we know that $r^{m-1} + \cdots + r + 1$ is congruent to $0 \pmod{p}$ and to $2 \pmod{4}$. Thus, constraints (\ref{eq:1}) and (\ref{eq:2}) hold; since $m = \phi(p)$, constraint (\ref{eq:3}) holds as well. We therefore have a construction for all but finitely many even-sized $p$-modular Hadamard matrices.

Now assume that $n$ is a quadratic residue. We can construct an $MH(n,p)$ by taking the matrix $2(D(MH(2n',p)) \oplus D_2) - J$, where $2n' \equiv n \pmod{p}$ as constructed above, and $D_2$ is a design of even size with appropriate parameters. For large enough $n$, this can be constructed if we find a design of odd size; using the same family of symmetric designs, we must find a design satisfying constraints (\ref{eq:2}),(\ref{eq:3}), and (\ref{eq:4}) above, in addition to the constraint that
\[1 + \frac{qr(r^m - 1)}{(r-1)} \equiv 0 \pmod{2}.\]
We can find $q$ and $r$ as above. If $r \equiv 1 \pmod{p}$, we set $m = p$ to satisfy all four constraints. Otherwise, we set $m = \frac{p-1}{2}$. Since $n$ is a quadratic residue, $r$ is one as well, so we still know that $r^m \equiv 1 \pmod{p}$. However, since $m$ is odd,$\frac{qr(r^m - 1)}{r-1}$ is also odd, and our design satisfies all four constraints.
\end{proof}

In particular, if $p$ is a prime satisfying Condition \ref{pps} with $p \equiv 3 \pmod{4}$ and such that $2$ is a primitive root of $p$, then Theorem \ref{3mod4} proves that Conjecture \ref{allth} holds for $p$. We use this to address the case when $p = 11$.

\begin{corollary}
Conjecture \ref{allth} is true for $p = 11$.
\end{corollary}
\begin{proof}
The following table of values displays the $q$ and $\delta$ values that show that $p=11$ satisfies Condition \ref{pps}, which proves Conjecture \ref{allth} when $p = 11$, and thus also the $11$-modular Hadamard conjecture for all but finitely many cases.

\begin{center}
\begin{tabular}{|c||c|c|}
\hline
$d$ & $q$ & $\delta$ \\
\hline
1 & 463 & 397 \\
2 & 397 & 13 \\
3 & 2663 & 3 \\
4 & 67 & 367 \\
5 & 23 & 5 \\
6 & 419 & 17 \\
7 & 947 & 7 \\
8 & 67 & 19 \\
9 & 617 & 317 \\
10 & 89 & 109 \\
\hline

\end{tabular}
\end{center}
\end{proof}

Meanwhile, if $p \equiv 1 \pmod{4}$, we can prove a weaker result.
\begin{theorem}
Let $p$ be a prime satisfying Condition \ref{pps} such that $2$ is a primitive root of $p$ and $p \equiv 2^i + 1 \pmod{2^{i+1}}$ for some integer $i \geq 1$. Let $0 \leq j \leq i$. Then, for all but finitely many $n$, if $4n^{-1} \equiv a^{2^j} \pmod{p}$ for some $a$ and if $n \equiv 0 \pmod{2^{i-j}}$, an $MH(n,p)$ exists.
\end{theorem}
\begin{proof}
We can again use the technique used in Theorem \ref{gentech} to construct matrices within a congruence class (mod $p$), again using designs from family 10 in \cite{ion06} of the form
\[\left(1 + \frac{qr(r^m-1)}{(r-1)},r^m,\frac{r^{m-1}(r-1)}{q}\right)\]
where $q$ and $r = \frac{q^d-1}{q-1}$ are both prime powers. The parameters of this design must satisfy constraints (\ref{eq:2}) through (\ref{eq:4}) of Theorem \ref{3mod4}, in addition to the following:
\[1+\frac{qr(r^m-1)}{(r-1)} \equiv 1 + 2^{i-j} \pmod{2^{i-j+1}}.\]
This equivalence will exactly ensure that all matrices with size congruent to $n$ and divisible by $2^{i-j}$ are eventually constructed, since the difference in size between any two matrices in our additive sequence will be also be congruent to $2^{i-j}$. Note as above that $n \equiv 4r^{-1} \pmod{p}$, so by assumption $r \equiv a^{2^j} \pmod{p}$ for some $a$.

First, we will address the case where $r \equiv 1 \pmod{p}$. Choosing $a$ to be $1$, $r \equiv a^{2^{i}} \pmod{p}$, so we will show that for all but finitely many $n \equiv 4 \pmod{p}$, where $4n^{-1} \equiv 1 \pmod{p}$, an $MH(n,p)$ exists. Thus, we want $\frac{qr(r^m-1)}{(r-1)}$ to be odd. Using $r$ and $q$ that satisfy Condition \ref{pps} with $\delta = 1$, and letting $m = p$, we satisfy constraints (\ref{eq:2}) through (\ref{eq:4}). Moreover, since $r$, $q$, and $m$ are all odd, $\frac{qr(r^m-1)}{(r-1)}$ is odd as well, so we are done.

Now assume that $r \not\equiv 1 \pmod{p}$. Again let $q$ and $r$ be those that satisfy Condition \ref{pps} with $\delta \equiv r \pmod{p}$, and set $m$ equal to the order of $r$ in the multiplicative group modulo $p$. Since $r \equiv a^{2^j} \pmod{p}$ and the order of $a$ divides $p-1$, we know that the highest power of $2$ dividing the order of $r$ must be $2^{i-j}$; equivalently, $m \equiv 2^{i-j} \pmod{2^{i-j+1}}$.
Since $q$ and $r$ are both odd, $\frac{qr(r^m-1)}{(r-1)}$ is divisible by as high a power of $2$ as $\frac{r^m-1}{r-1}$. We will show that $\frac{r^m-1}{r-1} \equiv 2^{i-j} \pmod{2^{i-j+1}}$.

We know according to Condition \ref{pps} that $r \equiv 1 \pmod{4}$, so we can fix $r_0$ such that $r = 4r_0 + 1$. Moreover,
\begin{align*}
\frac{r^m-1}{r-1} &= r^{m-1} + \cdots + 1 \\
&= \sum_{\alpha=0}^{m-1} \sum_{\beta=0}^{\alpha} {\alpha \choose \beta}(4r_0)^{\beta} \\
&= \sum_{\beta=0}^{m-1} \sum_{\alpha=\beta}^{m-1} {\alpha \choose \beta}(4r_0)^{\beta} \\
&= \sum_{\beta=0}^{m-1} (4r_0)^{\beta}\left({\beta \choose \beta} + \cdots + {m-1 \choose \beta}\right) \\
&= \sum_{\beta=0}^{m-1} (4r_0)^{\beta}{m \choose \beta + 1}.
\end{align*}
The $\beta = 0$ term is simply $m$, which is congruent to $2^{i-j} \pmod{2^{i-j+1}}$. Thus, if all other terms are congruent to $0 \pmod{2^{i-j+1}}$, we  will have that $\frac{r^m-1}{r-1} \equiv 2^{i-j} \pmod{2^{i-j+1}}$, as desired.\\
We therefore fix $\beta \geq 1$. We want $(4r_0)^{\beta}{m \choose \beta+1}$ to be divisible by $2^{i-j+1}$, where $2^{i-j}$ divides $m$. However,
\begin{align*}
(4r_0)^{\beta}{m \choose \beta + 1} &= r_0^\beta 2^{2\beta} \frac{m \cdot \cdots \cdot (m - \beta)}{(\beta+1)!},\\
\end{align*}
so it is sufficient to show that, once simplified and assuming $\beta \geq 1$, $\frac{2^{2\beta}}{(\beta+1)!}$ has an even numerator, in which case our expression will be divisible by $2m$ and thus by $2^{i-j + 1}$. The highest power of $2$ that divides $(\beta + 1)!$ is
\[\sum_{i=1}^{\infty}\left\lfloor \frac{\beta + 1}{2^i} \right\rfloor < \sum_{i=1}^{\infty} \frac{\beta + 1}{2^i} = \beta + 1 \leq 2\beta. \]

If $\beta > 1$, this is a strict inequality, so $\frac{2^{2\beta}}{(\beta + 1)!}$, once simplified, has an even numerator; if $\beta = 1$, then $\frac{2^{2\beta}}{(\beta + 1)!} = 2$. The expression $r_0^\beta 2^{2\beta} \frac{m \cdot \cdots \cdot m - \beta -1}{(\beta+1)!}$ must therefore be divisible by $2m$, and thus by $2^{i-j+1}$.
We then know that $\frac{r^m-1}{r-1} \equiv 2^{i-j} \pmod{2^{i-j+1}}$, as desired, so all but finitely many matrices of size divisible by $2^{i-j}$ and congruent to $n \pmod{p}$ will exist.
\end{proof}

\section{Nonexistence for small matrices}

The above results show conditionally that for certain $p$, for all but finitely many $n$, an $MH(n,p)$ exists if $n$ is even or a quadratic residue of $p$. We can show that when $n$ is small with respect to a modulus $m$, where $m$ is not necessarily prime, a stronger necessary condition for the existence of an $MH(n,m)$ holds.

\begin{theorem}
Let $n < 3m$ be odd and satisfy $(n,m) = 1$. Then an $MH(n,m)$ exists only if $m$ is odd and one of \[\frac{14m^2-mn-m-n^2+n \pm \sqrt{\Delta}}{8m}\] is a nonnegative integer, where \[\Delta = 36m^4+m^3(4-28n)+m^2(5n^2-2n+1) + 2mn(n^2-1) +(n-1)^2n^2.\] In particular, $\Delta$ must be a perfect square.
\label{onlybig}
\end{theorem}
\begin{proof}
If $m$ is even, we know by Lemma \ref{gcdmn} that no such matrix exists. Now assume we have an $MH(n,m)$ matrix $H$, where $m$ is odd, $n < 3m$, and $(n,m) = 1$. Let $p_{ij}$ be the inner product of the $i$th and $j$th rows, where $i \neq j$. Then, $p_{ij} \equiv 0 \pmod{m}$, but we also know that $-3m < p_{ij} < 3m$ and that $p_{ij}$ must be odd. Thus, $p_{ij} = \pm m$. 

We can then normalize $H$. After normalization, for all $i \neq 1$, either $p_{1i} = m$ and there are $\frac{n-m}{2}$ negative entries, or $p_{1i} = -m$, and there are $\frac{n+m}{2}$ negative entries. Note that since $(n,m) = 1$, $H^{\top}H = HH^{\top} = nI$, so each column also has either $\frac{n-m}{2}$ or $\frac{n+m}{2}$ negative entries. We will refer to rows or columns with $\frac{n-m}{2}$ negative entries as $\alpha$-rows or -columns, and to rows or columns with $\frac{n+m}{2}$ negative entries as $\beta$-rows or -columns. Note that there are exactly as many $\alpha$-rows as $\alpha$-columns. Moreover, in order to ensure an inner product of $\pm m$, any two $\beta$-rows must both have negative entries in exactly $\frac{3m+n}{4}$ columns, and given any $\alpha$-rows, any other row must have negative entries in exactly $\frac{n-m}{4}$ of the same columns.

We say that an $\alpha$-row designates $\frac{n-m}2$ columns by looking at the $\frac{n-m}2$ columns in which it has negative entries. For a given $\alpha$-row, let $a$ be the number of those columns that are $\alpha$-columns. Counting the total number of negative entries in these columns, first by columns and then by rows, we get the following relation:
\[a\left(\frac{n-m}2\right) + \left(\frac{n-m}2 - a\right)\left(\frac{n+m}2\right) = \frac{m+1}2 + \left(n-2\right)\left(\frac{n-m}4\right).\]
We can solve this to get that $a = \frac{n-m}4$.

Similarly, each $\beta$-row designates $\frac{n+m}2$ columns with a negative entry in that row. For a given $\beta$-row, let $b$ be the number of those columns that are $\beta$-columns, and let $c$ be the total number of $\beta$-columns, and thus of $\beta$-rows, in the matrix. Again counting the total number of negative entries in the designated $\frac{n+m}2$ columns, first by columns and then by rows, we arrive at the following:
\begin{multline*}
b\left(\frac{n+m}{2}\right) + \left(\frac{n+m}2 - b\right)\left(\frac{n-m}{2}\right) \\= \frac{3m+1}{2} + \left(c-1\right)\left(\frac{3m +n}4\right) + \left(n-c\right)\left(\frac{n-m}4\right).
\end{multline*}
Solving this equation, we get that $c-b = \frac{(n-m)(n-m-1)}{4m}$, so any $\beta$-row has positive entries in exactly $\frac{(n-m)(n-m-1)}{4m}$ $\beta$-columns.
Now let $d$ be the total number of columns of type $\frac{n-m}{2}$. Counting the number of negative entries in those columns gives us the relation
\[d\left(\frac{n-m}{2}\right) = d\left(\frac{n-m}{4}\right) + \left(2m-d\right)\left(\frac{n+m}{2} - \left(2m-d-\frac{(n-m)(n-m-1)}{4m}\right)\right).\]
We then solve to get that
\[d = \frac{14m^2-mn-m-n^2+n \pm \sqrt{\Delta}}{8m}\] is an integer, where \[\Delta = 36m^4+m^3(4-28n)+m^2(5n^2-2n+1) + 2mn(n^2-1) +(n-1)^2n^2.\]
Since we defined $d$ to be a number of rows of our matrix, $d$ must be a nonnegative integer, giving us the condition of this theorem.
\end{proof}

Lemma \ref{gcdmn} already shows that certain small Hadamard matrices cannot exist; for example, Lemma \ref{gcdmn} implies that no $MH(11,5)$ or $MH(6,5)$ can exist. However, Theorem \ref{onlybig} gives us a stronger constraint than had existed before; for example, it determines that no $MH(15,7)$ exists, despite satisfying the conditions of Lemma \ref{gcdmn}. Notably, the $MH(15,7)$ is left out of our construction above for matrices of the form $MH(7k+1,7)$. The $MH(15,7)$ is an example of one case when the above theorem has a particularly nice form; namely, when $n = 2m+1$, in which case the constraint states simply that $m^2 + (m+1)^2$ must be a perfect square.

The nonexistence of many modular Hadamard matrices of small size means that the ``for all but finitely many $n$" condition in Conjecture \ref{allth} is as strong as it can be; this can also be seen by noting that any $MH(n,m)$ with $n < m$ must be a real Hadamard matrix, and by referencing the constraint in Lemma \ref{gcdmn}.

\section*{Acknowledgements}
This research was supervised by Joe Gallian at the University of Minnesota, Duluth REU program, and funded by NSF grant DMS-1358569 and NSA grant H982230-13-1-0273. I would like to thank Joe Gallian for his advice and supervision in writing this paper, and Adam Hesterberg, Noah Arbesfeld, Daniel Kriz, and the referees for their helpful discussions, comments, and suggestions.

\bibliographystyle{amsplain}

\end{document}